\documentclass[12pt,leqno,a4paper]{amsart}
\usepackage{amssymb}

\newcommand{\FF}{{\mathbb{F}}}
\newcommand{\ZZ}{{\mathbb{Z}}}

\newcommand{\bC}{{\mathbf{C}}}
\newcommand{\bG}{{\mathbf{G}}}
\newcommand{\bH}{{\mathbf{H}}}
\newcommand{\bL}{{\mathbf{L}}}
\newcommand{\bN}{{\mathbf{N}}}
\newcommand{\bT}{{\mathbf{T}}}
\newcommand{\bw}{{\mathbf{w}}}

\newcommand{\fA}{{\mathfrak{A}}}
\newcommand{\fS}{{\mathfrak{S}}}

\newcommand{\cE}{{\mathcal{E}}}
\newcommand{\cF}{{\mathcal{F}}}

\newcommand{\ad}{{\operatorname{ad}}}
\newcommand{\Irr}{{\operatorname{Irr}}}
\newcommand{\PGL}{{\operatorname{PGL}}}
\newcommand{\PSL}{{\operatorname{PSL}}}
\newcommand{\GL}{{\operatorname{GL}}}
\newcommand{\SL}{{\operatorname{SL}}}
\newcommand{\GU}{{\operatorname{GU}}}
\newcommand{\SU}{{\operatorname{SU}}}
\newcommand{\PSp}{{\operatorname{S}}}
\newcommand{\Sp}{{\operatorname{Sp}}}
\newcommand{\GO}{{\operatorname{GO}}}
\newcommand{\SO}{{\operatorname{SO}}}
\newcommand{\Chevie}{{\sf Chevie}}

\newcommand{\tB}{{\tilde B}}
\newcommand{\tD}{{\tilde D}}
\newcommand{\tG}{{\tilde G}}
\newcommand{\tw}[1]{{}^{#1}\!}
\newcommand{\flr}[1]{{\lfloor #1\rfloor}}

\newcommand{\Ph}[1]{\Phi_#1}

\let\al=\alpha
\let\eps=\epsilon
\let\la=\lambda

\newtheorem{thm}{Theorem}[section]
\newtheorem{lem}[thm]{Lemma}
\newtheorem{prop}[thm]{Proposition}
\newtheorem{cor}[thm]{Corollary}

\newtheorem*{thmA}{Theorem 1}
\newtheorem*{thmB}{Theorem 2}
\newtheorem*{thmC}{Theorem 3}

\theoremstyle{remark}
\newtheorem{rem}[thm]{Remark}

\begin{document}

\title[Brauer's $k(B)$-conjecture]{On a minimal counterexample\\ to Brauer's $k(B)$-conjecture}

\date{\today}

\author{Gunter Malle}
\address{FB Mathematik, TU Kaiserslautern, Postfach 3049,
         67653 Kaisers\-lautern, Germany.}
\email{malle@mathematik.uni-kl.de}

\thanks{The author gratefully acknowledges financial support by SFB TRR 195.}

\keywords{Number of simple modules, Brauer's $k(B)$-conjecture, blocks of simple groups}

\subjclass[2010]{20C15, 20C33}

\begin{abstract}
We study Brauer's long-standing $k(B)$-conjecture on the number of characters
in $p$-blocks for finite quasi-simple groups and show that their blocks do not
occur as a minimal counterexample for $p\ge5$ nor in the case of abelian
defect. For $p=3$ we obtain that the principal 3-blocks do not provide minimal
counterexamples. We also determine the precise number of irreducible characters
in unipotent blocks of classical groups for odd primes.
\end{abstract}

\maketitle


\section{Introduction} \label{sec:intro}

In his 1954 address at the International Congress in Amsterdam Richard Brauer
posed a list of fundamental problems in representation theory
of finite groups \cite{Br54}, many of which are still open. One among them is
his conjecture on the number $k(B)$ of irreducible complex characters in a
$p$-block $B$ of a finite group: this number $k(B)$ should be at most equal
to the order $|D|$ of a defect group $D$ of $B$. Brauer and Feit \cite{BF59}
already showed that $k(B)\le p^{2d-2}$ if $|D|=p^d$, but to the present day,
no general upper bound linear in $|D|$ is known for $k(B)$. For $p$-solvable
groups the $k(B)$-conjecture was shown by Nagao to reduce to the coprime
$k(GV)$-problem, which was finally settled in 2004, see \cite{kGV}. In fact,
using this result Robinson \cite[Thm.~1]{Ro04} showed that for $p$-solvable
groups we even have $k(B)<|D|$ whenever $D$ is non-abelian.

This motivates to consider the following strong form of Brauer's
$k(B)$-conjecture:
\begin{itemize}
  \item[\qquad\qquad\qquad]
 \emph{Let $B$ be a block of a finite group with defect group $D$.\\
  Then $k(B)\le |D|$, with strict inequality unless $D$ is abelian.}
\end{itemize}

In this paper we investigate a possible minimal counterexample to this
conjecture, which by the results cited above must necessarily be non-solvable.
This focuses attention on the non-abelian simple groups and their covering
groups. Our main result is:

\begin{thmA}
 Let $p\ge5$ be a prime and $B$ a $p$-block of a finite quasi-simple
 group $G$. Then $B$ is not a minimal counterexample to the strong form of
 Brauer's $k(B)$-conjecture.
\end{thmA}

Here, $(G,B)$ is a called a minimal counterexample if the conjecture holds
for all $p$-blocks $B_1$ of groups $G_1$ with $|G_1/Z(G_1)|$ strictly
smaller than $|G/Z(G)|$ having defect groups isomorphic to those of $B$. We
also show that blocks with abelian defect cannot lead to minimal
counterexamples (in view of Theorem~1, this concerns the primes $p=2,3$):

\begin{thmB}
 Let $p$ be a prime and $B$ a $p$-block of a finite quasi-simple group $G$
 with abelian defect groups. Then $B$ is not a minimal counterexample to
 Brauer's $k(B)$-conjecture.
\end{thmB}

We also obtain strong restrictions for the primes $p=2,3$:

\begin{thmC}
 Assume that a $p$-block $B$ of a finite quasi-simple group $G$ is a minimal
 counterexample to Brauer's $k(B)$-conjecture. Then $p\le3$, $G$ is of Lie
 type in characteristic not $p$, $B$ is an isolated block of $G$, the defect
 groups of $B$ are non-abelian, and either $p=2$ or $B$ is not unipotent.
 In particular, the principal 3-block is not a minimal counterexample.
\end{thmC}

The proofs will be given in the subsequent sections, using the classification
of finite simple groups in conjunction with Lusztig's theory of characters of
finite reductive groups. While our methods fall short of verifying the
$k(B)$-conjecture for all blocks of quasi-simple groups, since they partly
rely on Bonnaf\'e--Rouquier type reduction arguments to rule out a minimal
counterexample, in the most interesting case of unipotent blocks with $p>2$
and non-abelian defect our arguments actually show that these do satisfy the
$k(B)$-conjecture.

\begin{rem}   \label{rem:rest}
Let us comment on the cases left open by our results. For $p=2,3$ the isolated
blocks of all quasi-simple groups of Lie type remain to be considered. For
$p=2$, already the case of unipotent blocks of $\SL_n(q)$ seems hard.
\end{rem}

\begin{rem}   \label{rem:abelian}
If $B$ is a block with abelian defect group $D$, then according to (the
proven direction of) Brauer's height zero conjecture all characters in
$\Irr(B)$ are of height zero. By the Alperin--McKay conjecture they should be
in bijection with the height zero characters of the Brauer corresponding block
$b$ of $N_G(D)$ (which also has defect group $D$), whence $k(B)=k(b)$. Thus,
assuming the validity of the Alperin--McKay conjecture, a block $B$ with
non-normal abelian defect group cannot be a minimal counterexample to the
$k(B)$-conjecture. The interesting situation hence rather seems to be the one
of non-abelian defect groups; see also the recent result of Sambale recalled in
Theorem~\ref{thm:sam}.
\end{rem}

No reduction of the general conjecture to the case of (quasi-)simple groups
has been found so far; see Navarro's article \cite{Na17} for some thoughts and
ideas in that direction.
\medskip

The paper is built up as follows: In Section~\ref{sec:non-lie} we settle the
case of quasi-simple groups not of Lie type, mostly by collecting results from
the literature. In Section~\ref{sec:lie-p} we prove our main theorem for groups
of Lie type in their defining characteristic. The by far most complicated case,
groups of Lie type in cross characteristic, is then considered, after some
general reductions in Section~\ref{sec:lie-nondef}, in
Section~\ref{sec:class} where we deal with classical groups at odd primes,
and in Section~\ref{sec:exc} where we concentrate on groups of exceptional
type. On the way we also extend the results of Olsson \cite{Ol84}
to derive explicit formulas for the number of characters in unipotent blocks,
which may be of independent interest (see Propositions~\ref{prop:kB BC}
and~\ref{prop:kB D}). The proof of Theorems~1, 2 and~3 is given at the end of
Section~\ref{subsec:E8}.
\medskip

\noindent{\bf Acknowledgement:} I thank Gabriel Navarro for helpful
correspondence on the topic of the paper and Jay Taylor for his remarks on an
earlier version.

\section{Groups not of Lie type} \label{sec:non-lie}

In this section we consider the quasi-simple groups $G$ such that $G/Z(G)$ is
not of Lie type. Here, and later on, we will make use of the following result
\cite[Thm.~A]{S17}:

\begin{thm}[Sambale]   \label{thm:sam}
 The $k(B)$-conjecture holds for all blocks $B$ with abelian defect groups
 of rank at most~3.
\end{thm}

Brauer had previously shown his conjecture for abelian defect groups of rank
at most~2. Moreover, we have the following reduction, relying on a result
of Navarro:

\begin{thm}[Navarro]   \label{thm:nav}
 The $k(B)$-conjecture (in strong form) holds for all $p$-blocks of a
 quasi-simple group $G$ if it holds for the faithful $p$-blocks of
 $p'$-covering groups of $G/Z(G)$.
\end{thm}

\begin{proof}
Let $B$ be a $p$-block of $G$. Set $G_1=G/O_p(G)$ and let $B_1$ be the unique
$p$-block of $G_1$ contained in $B$. Then by \cite[Thm.~C]{Na17}, $B$ satisfies
the $k(B)$-conjecture if $B_1$ does. So it suffices to consider $p'$-coverings
of $G/Z(G)$. For $G$ such a covering group let $K$ be the intersection of the
kernels of the characters in $B$ and $B_1$ the $p$-block of $G_1=G/K$ contained
in $B$. Then we have $k(B)=k(B_1)$ and the defect groups are isomorphic as $K$
is a $p'$-subgroup, so again the validity of the $k(B)$-conjecture for $B_1$
implies it for~$B$.
\end{proof}

\begin{prop}   \label{prop:spor}
 Let $G$ be quasi-simple such that $S=G/Z(G)$ is one of:
 \begin{enumerate}
  \item[\rm(1)] a sporadic simple group;
  \item[\rm(2)] an alternating group $\fA_6$ or $\fA_7$;
  \item[\rm(3)] an exceptional covering group of a simple group of Lie type; or
  \item[\rm(4)] the Tits group $\tw2F_4(2)'$.
 \end{enumerate}
 Then the $k(B)$-conjecture in the strong form holds for all $p$-blocks of $G$
 for all primes $p$.
\end{prop}

\begin{proof}
The character tables of all the groups $G$ in the statement are known and
available in GAP \cite{GAP}; since the block-subdivision of $\Irr(G)$ and
the defect of a block can be computed from a knowledge of the character
table, the claim can easily be verified automatically. In fact, by
Theorem~\ref{thm:sam} the $k(B)$-conjecture holds for blocks with
abelian defect of rank at most~3, so only very few cases remain to be
considered.
\end{proof}

\begin{thm}[Olsson]
 Let $G$ be a covering group of the alternating group $\fA_n$, $n\ge5$. Then
 the $k(B)$-conjecture in the strong form holds for all $p$-blocks of
 $G$ for all primes~$p$.
\end{thm}

\begin{proof}
By Proposition~\ref{prop:spor} we may assume that $n\ne6,7$. The case of
$G=\fA_n$, as well as of its double cover $\tilde\fA_n$ when $p\ne2$ is due to
Olsson \cite{Ol90}. The validity for the faithful 2-blocks of the double cover
now follows from Theorem~\ref{thm:nav}.
\end{proof}

\section{Groups of Lie type in defining characteristic} \label{sec:lie-p}

In this section we show that the $p$-blocks of quasi-simple groups of Lie type
in characteristic~$p$ satisfy the $k(B)$-conjecture.

\begin{prop}   \label{prop:A1}
 Let $G$ be a covering group of $\PSL_2(q)$, with $q=p^f\ge4$. Then the
 $k(B)$-conjecture (in the strong form) holds for all $p$-blocks of $G$.
\end{prop}

\begin{proof}
By Proposition~\ref{prop:spor} we may assume that $G$ is not an exceptional
covering group of $S:=\PSL_2(q)$. Note that Sylow $p$-subgroups of $G$ are
(elementary) abelian. First assume that $p=2$. Then $G=S$, $|\Irr(G)|=q+1$,
and all but the Steinberg character of $G$ lie in the principal block $B$ of
$G$, with defect groups the Sylow $p$-subgroups of $S$ of order $q$. So we
have equality $k(B)=|D|$. 
\par
Now assume that $p>2$. Here we may assume that $G=\SL_2(q)$, since
$|Z(\SL_2(q))|=2$ is prime to~$p$. Now $G$ has three $p$-blocks, one of
defect zero, and the other two lying above the two central characters of
$G$. The latter have full defect and each contain $(q+3)/2$ irreducible
characters, while Sylow $p$-subgroups of $G$ have order~$q$.
\end{proof}

\begin{thm}   \label{thm:lie-p}
 Let $G$ be quasi-simple of Lie type in characteristic~$p$. Then the
 $k(B)$-conjecture in the strong form holds for all $p$-blocks of $G$.
\end{thm}

\begin{proof}
Again by Proposition~\ref{prop:spor} we may assume that $G$ is not an
exceptional covering group of $S:=G/Z(G)$. Thus, $|Z(G)|$ is prime to~$p$
and hence we may assume that $G$ is the universal $p'$-covering group of $S$,
whence it can be obtained as the group of fixed points under a Steinberg map
$F$ of a simple simply connected linear algebraic group $\bG$ over an
algebraic closure of $\FF_p$. Let $q$ be the positive real number such that
some power $F^d$ of $F$ acts as $q^d$ on the character group of an $F$-stable
maximal torus of $\bG$. The order of a Sylow $p$-subgroup of $G$ is then
given by $q^N$, where $N$ is the number of positive roots in the root system
of $\bG$ (see e.g.~\cite[Prop.~24.3]{MT}).
By a theorem of Humphreys \cite{Hum} any $p$-block of $G$ either has full
defect, or defect zero. We now employ results on upper bounds for the number
of conjugacy classes $k(G)$. By
Fulman and Guralnick \cite[Thm.~1.1]{FG12} we have $k(G)\le 27.2\,q^r$,
where $r$ is the rank of $\bG$. Comparing with the size $q^N$ of a Sylow
$p$-subgroup we see that our claim holds whenever $q^{N-r}>27.2$, so
certainly if $N-r\ge5$. Since $G$ is not of type $A_1$ by
Proposition~\ref{prop:A1}, this only leaves the groups of types $A_2$, $A_3$,
$\tw2A_2$, $\tw2A_3$, $B_2$, $G_2$, $\tw2B_2$ and $^2G_2$. For the latter three
series of exceptional groups, \cite[Table~1]{FG12} gives the precise value of
$k(G)$, from which our desired inequality readily follows. For $G=\SL_3(q)$
or $\SU_3(q)$ we have $k(G)\le q^2+q+8$ (see \cite[p.~3032]{FG12}), while
Sylow $p$-subgroups have order $q^3$, so only the case $q=2$ remains. But
$k(\SL_3(2))=6$, while $\SU_3(2)$ is solvable. For types $A_3$ and $\tw2A_3$
we have $N-r=3$, thus we are done when $q^3>27.2$, that is, for $q>3$. But
$k(\SL_4(2))=14<2^6$, $k(\SU_4(2))=20<2^6$, $k(\SL_4(3))=51<3^6$,
$k(\SU_4(3))=71<3^6$.
\par
Finally, for $G$ of type $B_2$ we have $N-r=2$, so we only need to check
$q\le5$, which is readily done with GAP.
\end{proof}

\section{Groups of Lie type in non-defining characteristic} \label{sec:lie-nondef}
In this section we start our investigation of groups of Lie type for
non-defining primes.
We consider the following setup. Let $\bG$ be a simple algebraic group of
simply connected type over an algebraic closure of $\FF_p$, with a Steinberg
map $F:\bG\rightarrow\bG$. We write $G:=\bG^F$ for the group of fixed
points, a finite group of Lie type, which is quasi-simple except in finitely
many cases. Throughout, $\ell$ is a prime dividing $|G|$ but different
from~$p$.


\begin{lem}   \label{lem:suzree}
 The $k(B)$-conjecture holds in its strong form for all blocks of the Suzuki
 and Ree groups.
\end{lem}

\begin{proof}
For these groups, Sylow $\ell$-subgroups for $\ell\ne p$ are abelian of rank
at most~3, hence the conjecture holds by Theorems~\ref{thm:lie-p}
and~\ref{thm:sam}, unless $\ell=3$ and $G=\tw2F_4(2^{2f+1})$. The assertion
in the latter case was shown in \cite[Kor.]{M2F4}.
\end{proof}

From now on we may hence assume that $G$ is not a Suzuki or Ree group and so
that $F$ is a Frobenius map with respect to some $\FF_q$-structure, with
$q$ a power of $p$. For a prime $\ell$ different from $p$ we denote by
$d_\ell(q)$ the order of $q$ modulo~$\ell$, if $\ell$ is odd, respectively
the order of $q$ modulo~4 if $\ell=2$.

We recall some results on the block distribution of characters of finite
reductive groups. This is closely related to the partition of $\Irr(G)$ into
Lusztig series $\cE(G,s)$, where $s$ runs over semisimple elements of the dual
group $G^*:=\bG^{*F}$ modulo conjugation. Let $B$ be an $\ell$-block of $G$.
Then there exists a semisimple $\ell'$-element $s\in G^*$ such that $B$ is
contained in
$$\cE_\ell(G,s):=\coprod_{t} \cE(G,st),\eqno{(*)}$$
where $t$ runs over a system of representatives of conjugacy classes of
semisimple $\ell$-elements in $C_G(s)$ (see \cite[Thm.~9.12]{CE}). Recall that
a semisimple element $s\in\bG^{*F}$ is called \emph{isolated} if the
connected component $C_{\bG^*}^\circ(s)$ of the centraliser $C_{\bG^*}(s)$
is not contained in any proper $F$-stable Levi subgroup $\bL^*< \bG^*$.

We now argue that most $\ell$-blocks of $G$ are Morita equivalent to
blocks of smaller groups and so cannot give rise to a minimal counterexample.

\begin{lem}   \label{lem:not qi}
 Let $s\in G^*$ be a non-isolated semisimple $\ell'$-element. Then no
 $\ell$-block of $G$ in $\cE_\ell(G,s)$ is a minimal counterexample to the
 strong form of the $k(B)$-conjecture.
\end{lem}

\begin{proof}
By assumption there is some proper $F$-stable Levi subgroup
$\bL^*<\bG^*$ such that $C_{\bG^*}^\circ(s)\le\bL^*$. Then by a deep result
of Bonnaf\'e, Dat and Rouquier \cite[Thm.~7.7]{BDR15}, any $\ell$-block
$B$ in $\cE_\ell(G,s)$ is Morita equivalent to a suitable $\ell$-block $b$ of
$\bN^F$, where $\bN\ge \bL$ is dual to $C_{\bG^*}(s)^F\bL$. Moreover this
is induced by Jordan decomposition, hence compatible with central characters.
Then $k(B)=k(b)$ and the defect groups of $B$ and $b$ have the same order
(and are in fact isomorphic), whence the $k(B)$-conjecture holds for $B$ by
the minimality assumption. 
\end{proof}

We are thus left with the situation that $B$ lies in a series $\cE_\ell(G,s)$
with $s\in G^*$ an isolated semisimple element.
In this case we can utilise a result of Enguehard:

\begin{prop}   \label{prop:Enguehard}
 Let $B$ be an isolated, non-unipotent $\ell$-block of a quasi-simple group
 of Lie type $H$ for a prime $\ell\ge3$ that is good for $H$. Then $B$ is not
 a minimal counterexample to the strong form of the $k(B)$ conjecture.
\end{prop}

\begin{proof}
By Proposition~\ref{prop:spor} we may assume that $H$ is not an exceptional
covering group. Thus, $H=G/Z$, where $G=\bG^F$ is as above and $Z\le Z(G)$. By
Lemma~\ref{lem:not qi} we may assume that $\bG$ is not of type $A$ as the
only isolated element in type $A$ is the identity, which corresponds to the
unipotent blocks. But then $\ell$ good implies that $\ell$ does not divide
$|Z(G)|$, so we may consider $B$ as an $\ell$-block of $G$. By the main result
of Enguehard \cite[Thm.~1.6]{En08} there is a group $G_1$, with $|G_1/Z(G_1)|$
strictly smaller than $|G/Z(G)|$ (since $B$ is not unipotent) with an
$\ell$-block $B_1$ having the same invariants (number of irreducible
characters and defect group) as $B$. Moreover, there is a bijection
$\Irr(B)\rightarrow\Irr(B_1)$ preserving central characters. In particular,
$B$ and any block of $H$ dominated by $B$ is not a minimal counterexample.
\end{proof}

We note one further reduction which will be used for isolated 5-blocks of
$E_8(q)$:

\begin{lem}   \label{lem:one block}
 Let $s\in G^*$ be a non-central semisimple $\ell'$-element with connected
 centraliser $\bC^*=C_{\bG^*}(s)$ such that $\cE_\ell(G,s)$ is a single
 $\ell$-block $B$, and $\cE_\ell(C,1)$ also is a single $\ell$-block $b$,
 where $\bC$ is dual to $\bC^*$. Then $B$ is
 not a minimal counterexample to the strong form of the $k(B)$ conjecture.
\end{lem}

\begin{proof}
We have that $\Irr(B)=\coprod_t\cE(G,st)$ and $\Irr(b)=\coprod_t\cE(C,t)$
where both disjoint unions run over $C^*:=C_{G^*}(s)$-conjugacy classes of
$\ell$-elements $t$ in $C^*$. Application of Jordan decomposition in $G$ as
well as in $C$ puts both $\cE(G,st)$ and $\cE(C,t)$ into bijection with the
same Lusztig series $\cE(C_t,1)$ of $C_t:=C_{G^*}(st)=C_{C^*}(t)$ for any
such $t$. Thus, if $\chi\in\cE(G,s)$ and $\chi'\in\cE(C,t)$ correspond to
the same character in $\cE(C_t,1)$ we have
$\chi(1)/\chi'(1)=|G^*:C_{G^*}(s)|_{p'}$. So comparing character degrees of
height zero characters in the two Lusztig series we see that $B$ and $b$ have
defect groups of the same order. Thus the inequality holds for $B$ if (and
only if) it holds for $b$.
\end{proof}

For unipotent blocks, that is, $\ell$-blocks in $\cE_\ell(G,1)$, we have the
following partial reduction:

\begin{lem}   \label{lem:ab def}
 Let $\ell$ be a good prime for $\bG$ which does not divide $|Z(G)|$, and
 let $B$ be a unipotent $\ell$-block of $G/Z$ with abelian defect groups, where
 $Z\le Z(G)$. Then $B$ is not a minimal counterexample to the strong form of
 the $k(B)$-conjecture.
\end{lem}

\begin{proof}
By our assumption on $\ell$ we may assume that $B$ is a unipotent block of $G$.
Then under the stated conditions, by \cite[Thm.~3.1(2)]{BM93} the block $B$
is isotypic with a block of the normaliser of a defect group, and hence
can't be a counterexample by minimality.
\end{proof}

To deal with unipotent blocks with non-abelian defect groups we will need some
information on conjugacy classes of $\ell$-elements and their Lusztig series.
The following was essentially shown by Cabanes and Enguehard:

\begin{prop}   \label{prop:isogeny}
 Let $\sigma:\bG\rightarrow\bH$ be an isogeny of connected reductive groups
 in characteristic~$p$ equivariant with respect to Frobenius endomorphisms
 $F,F'$ of $\bG,\bH$ respectively, with dual isogeny
 $\sigma^*:\bH^*\rightarrow\bG^*$. Set $r:=|Z(\bG)/Z^\circ(\bG)|$. Let $\ell$
 be a prime different from $p$ not dividing $r$ and $s\in H^*$ of order prime
 to $r\ell$. Then $\sigma^*$ induces a bijection between the conjugacy classes
 of $\ell$-elements in $C_{H^*}(s)$ and $C_{G^*}(\sigma^*(s))$, and for any
 such $\ell$-element $t$ we have a bijection
 $\cE(H,st)\rightarrow\cE(G,\sigma^*(st))$.
\end{prop}

\begin{proof}
First consider the case that $\sigma:\bG\rightarrow\bG_\ad$ is the adjoint
quotient map. Then the claim follows from \cite[Prop.~17.4]{CE}. Application
of the same result to the adjoint quotient
$\sigma':\bH\rightarrow\bH_\ad=\bG_\ad$, and then composing the obtained
bijections gives the statement.
\end{proof}

\section{Classical groups in non-defining characteristic} \label{sec:class}
In this section we consider unipotent $\ell$-blocks of classical groups for
odd primes $\ell$ different from their defining characteristic.

\subsection{Linear and unitary groups} \label{subsec:lin}
We start by recalling the parametrisation of conjugacy classes of
$\ell$-elements in
$\GL_n(q)$. Let $\ell$ be an odd prime not dividing $q$. We write $\cF_\ell$ for
the set of irreducible monic polynomials over $\FF_q$ whose roots have
$\ell$-power order in $\overline{\FF_q}^\times$. Let $d=d_\ell(q)$ be the order
of $q$ modulo~$\ell$, and let $\ell^a$ be the precise power of $\ell$ dividing
$q^d-1$. For $f\in \cF_\ell$ we set $c_f=\max\{0,i-a\}$ if the roots of $f$
have order $\ell^i$. Write $n=wd+r$ with $0\le r<d$. With this notation the
conjugacy classes of $\ell$-element in $\GL_n(q)$ are parametrised by
\emph{$\ell$-weight vectors of $w$}, that is, by functions
$$m:\cF_\ell\rightarrow\ZZ_{\ge0},\quad f\mapsto m_f,\qquad\text{ with }
  \sum_{f\in\cF_\ell}m_f\ell^{c_f}=w.$$
The corresponding $\ell$-elements have characteristic polynomial
$(X-1)^{dm_1+r}\prod_{f\ne X-1} f^{m_f}$ in $\GL_n(q)$, where we have set
$m_1:=m_{X-1}$. The centraliser of such an element is then a direct product
$\GL_{dm_1+r}(q)\prod_{f\ne X-1}\GL_{m_f}(q^{d\ell^{c_f}})$. 

To treat the special linear and unitary groups we need to determine the number
$k(B)$ for their principal $\ell$-block $B$; to do this, we use of the result of
Olsson \cite{Ol84} who verified the $k(B)$-conjecture for all $\ell$-blocks of
$\GL_n(q)$ and $\GU_n(q)$, where $2<\ell\ne p$.

For integers $s,t\ge1$ let $k(s,t)$ denote the number of $s$-tuples of
partitions of $t$, and for $\ell,a,w\ge1$ let
$$k(\ell,a,w)
  := \sum_\bw k(\ell^a,w_0)\prod_{i\ge1}k(\ell^a-\ell^{a-1},w_i)$$
where the sum runs over all \emph{$\ell$-compositions of $w$}, that is, all
tuples $\bw=(w_0,w_1,\ldots)$ of non-negative integers satisfying
$\sum_{i\ge0} w_i\ell^i =w$. As customary, we write $\GL_n(-q):=\GU_n(q)$
and $\SL_n(-q):=\SU_n(q)$.

\begin{thm}   \label{thm:k(B) SLn}
 Let $\SL_n(\eps q)\le G\le\GL_n(\eps q)$ with $\eps\in\{\pm1\}$, and $\ell>2$
 be a prime dividing $q-\eps$. Set $\ell^a:=(q-\eps)_\ell$,
 $\ell^m:=|Z(\SL_n(\eps q))|_\ell=\gcd(n,q-\eps)_\ell$,
 $\ell^g:=|\GL_n(\eps q):G|_\ell$ and $u:=\min\{m,g\}$.
 Let $B$ denote the principal $\ell$-block of $G$. Then
 $$k(B)=\Big(k(\ell,a,n)
     +\sum_{i=1}^u \ell^{2i-2}(\ell^2-1)\,k(\ell,a,n/\ell^i)\Big)/\ell^g.$$
\end{thm}

\begin{proof}
Let first $\eps=1$, so $\SL_n(q)\le G\le\GL_n(q)=:\tG$. Let $\tilde B$ be
the principal $\ell$-block of $\tG$; thus $\Irr(\tB)$ is the
union of the Lusztig series $\cE(\tG,t)$ where $t$ runs over $\ell$-elements
of $\tG^*=\GL_n(q)$. Then the characters in the principal $\ell$-block $B$ of
$G$ are precisely the constituents of the restrictions to $G$ of the
$\chi\in\Irr(\tilde B)$. Now by Lusztig's Jordan decomposition an irreducible
character $\chi$ lying in the Lusztig series $\cE(\tG,t)$ of a semisimple
element $t\in\tilde G^*$ restricts irreducibly to $\SL_n(q)$ unless
the image $\bar t$ of $t$ in $\PGL_n(q)$ has disconnected centraliser in
$\PGL_n$. Now assume that $t\in\GL_n(q)$ is an $\ell$-element. If the
centraliser of $\bar t$ in $\PGL_n$ is disconnected, then $\gcd(n,q-1)$ and
$o(t)$ are not coprime (see \cite[Prop.~14.20]{MT}), that is, $\ell$
divides~$n$. In this case, $\chi\in\cE(\tG,t)$ splits into precisely
$A_t:=|C_{\PGL_n}(\bar t)^F:C_{\PGL_n}^\circ(\bar t)^F|$ distinct irreducible
characters of $\SL_n(q)$, where $F$ is the standard Frobenius endomorphism on
$\PGL_n$. By Clifford's theorem this implies that $\chi$ splits into exactly
$\ell^u=\gcd\{|\tG:G|,\ell^m\}$ distinct constituents upon restriction to~$G$.
\par
Thus in order to calculate $k(B)$ we need to determine the
centraliser of $\bar t$ for $\ell$-elements $t\in\tG$. Let $i\ge1$. Then
$A_t\ge\ell^i$ if and only if $t$ and $\zeta t$ have the same eigenvalues for
an $\ell^i$th root of unity $\zeta\in\FF_q^\times$. That is, the characteristic
polynomial $f\in\FF_q[X]$ of $t$ satisfies $f(X)=f(\zeta X)$. Now multiplication
by $\zeta$ makes orbits of length $\ell^i$ on $\overline{\FF}_q^\times$, and
such an orbit is a subset of the
roots of a polynomial $f\in\cF_\ell$ if and only if $\deg(f)$ is divisible by
$\ell^i$. Thus, the $\ell$-elements $t$ with $\ell^i|A_t$ are parametrised
by maps $m:\cF_\ell\rightarrow\ZZ_{\ge0}$ as above that are constant on
$\zeta$-orbits. Arguing as in \cite[proof of Prop.~6]{Ol84} we see that
the number of irreducible characters of $\tG$ in the union of the corresponding
Lusztig series is $k(\ell,a,n/\ell^i)$. The claim now follows
by induction over the $\ell$-part $\ell^g$ of the index of $G$ in $\tG$.
\end{proof}

\begin{lem}   \label{lem:p(w)}
 Let $\ell\ge2$ and $w\ge1$. The number $p_\ell(w)$ of $\ell$-compositions of
 $w$ (partitions into parts of $\ell$-power order) satisfies
 $p_\ell(w)\le \ell^{\binom{u+1}{2}}$ where $u=\flr{\log_\ell w}$.
\end{lem}

\begin{proof}
The claim is easily verified for $w<\ell^2$. Else, arranging the
sought for partitions according to their number of parts of length~1 we indeed
get, arguing by induction
$$p_\ell(w)\le\sum_{j=0}^{\flr{w/\ell}}p_\ell\big(\flr{w/\ell}-j\big)
  \le \frac{w}{\ell}\,p_\ell(\flr{w/\ell})\le \prod_{i=1}^u \frac{w}{\ell^i}
  =\frac{w^u}{\ell^{\binom{u+1}{2}}} \le\ell^{\binom{u+1}{2}}.\qedhere$$
\end{proof}

\begin{thm}   \label{thm:SLn}
 Let $H=G/Z$ with $G\in\{\SL_n(q),\SU_n(q)\}$ and $Z\le Z(G)$. Let $\ell>2$ be
 a prime. Then the unipotent $\ell$-blocks of $H$ are not counterexamples to
 Brauer's $k(B)$-conjecture.
\end{thm}

\begin{proof}
We embed $G:=\SL_n(q)\unlhd\tilde G:=\GL_n(q)$. Let $B$ be a unipotent
$\ell$-block of $G$ with defect group $D$. Then there is a unipotent
$\ell$-block $\tB$ of $\tilde G$ covering $B$, with defect group $\tD\ge D$,
and by \cite[Thm.]{Ol84} we have $k(\tB)\le |\tD|$. As $\tG/G$ is cyclic,
restriction of characters from $\tilde G$ to $G$ is multiplicity-free. 
Moreover, as seen in the previous proof, an irreducible character
$\chi\in\Irr(\tB)$ restricts irreducibly to $G$ unless $\ell$ divides
$r:=\gcd(n,q-1)$.
\par
First assume that $\ell$ does not divide $r$. Then all $\chi\in\Irr(\tB)$
restrict irreducibly to $G$ so $k(B)=k(\tB)/\ell^a$. As $|D|=|\tD|/\ell^a$ the
conjecture follows for the block $B$ of $\SL_n(q)$. Furthermore, $|Z(G)|=r$ is
not divisible
by $\ell$ in this case, so all characters in $\Irr(B)$ have $Z(G)$ in their
kernel, and the claim also follows for $H=G/Z$ for any $Z\le Z(G)$.
\par
So now assume that $\ell|\gcd(n,q-1)$, and write $\ell^a$ for the precise power
of $\ell$ dividing $q-1$, and $\ell^m$ for the $\ell$-part of $|Z(\SL_n(q))|$.
Then all unipotent characters of $G$ lie in the
principal $\ell$-block (see e.g.~\cite[Thm.]{CE94}), so the defect group $D$ of
$B$ is a Sylow $\ell$-subgroup of $G$. From Theorem~\ref{thm:k(B) SLn} we now
get, arguing as in the proof of \cite[(I)]{Ol84}
$$k(B)
 \le\Big(p_\ell(n)\ell^{an}+\sum_{i=1}^mp_\ell(n/\ell^i)\ell^{an/\ell^i+2i}\big)
 /\ell^a
 \le p_\ell(n)\big(\ell^{an}+\sum_{i=1}^m\ell^{an/\ell^i+2i}\big)/\ell^a.
$$
First assume that $n\ge2\ell$. Then $an/\ell^i+2i\le an-2i$ for $i\ge1$ and so
by Lemma~\ref{lem:p(w)}
$$k(B)\le p_\ell(n)\ell^{an-a+1/(\ln(\ell)(\ell^2-1))}
  \le \ell^{an-a+1/(\ln(\ell)(\ell^2-1))+\binom{u+1}{2}}$$
with $u=\flr{\log_\ell(n)}$. As we have $|D|=\ell^{an-a+(n!)_\ell}$ and
$|\bar D|=\ell^{an-a+(n!)_\ell-m}$ where $m\le u$, the claim follows
except when $n=2\ell$ or $\ell=3$, $n\le12$. In the first case we use that
$p_\ell(2\ell)=3$, and the latter cases can be checked individually.
\par
Finally assume that $n=\ell$. Then by \cite[p.~46]{Ol84} we have
$k(\tilde B)\le\ell^{a\ell}+\ell^a$ and by Theorem~\ref{thm:k(B) SLn},
$k(B)\le \ell^{a\ell-a}+\ell^2<\ell^{a\ell-a+1}=|\tilde D|/\ell^a=|D|$.
Furthermore, a character parametrised by a semisimple $\ell$-element
$t\in\GL_n(q)$ is trivial on the center if and only if $t$ lies in the derived
subgroup $\SL_n(q)$. From this it can then be checked
that the principal $\ell$-block $\bar B$ of $\PSL_\ell(q)$ has
$k(\bar B)=(k(B)+\ell-1))/\ell\le |D|/\ell=|\bar D|$, with
$\bar D=D/Z(\SL_\ell(q))$ a Sylow $\ell$-subgroup of $\PSL_\ell(q)$, so our
claim follows for $H=G/Z(G)$ as well.
\par
The arguments for $G=\SU_n(q)$ are entirely analogous, again relying on the
explicit formula for $k(\tilde B)$ in \cite[Prop.~6]{Ol84} and for $k(B)$ in
Theorem~\ref{thm:k(B) SLn}.
\end{proof}

\subsection{Classical groups}   \label{subsec:class}
We now turn to the quasi-simple groups of symplectic and orthogonal type.
Here, in the spirit of Olsson's result for linear and unitary groups, we first
derive a formula for the number of characters in unipotent blocks which
may be of independent interest.    \par
Let $G_n(q)$ be one of $\Sp_{2n}(q)$ or $\SO_{2n+1}(q)$ and $\ell\ne2$ an odd
prime not dividing $q$. We write $d=d_\ell(q)$ for the order of $q$
modulo~$\ell$
and let $d':=d/\gcd(d,2)$. The unipotent $\ell$-blocks of $G_n(q)$ are
parametrised by $d$-cuspidal pairs in $G_n(q)$, that is, by pairs $(L,\la)$
where $L=G_{n-wd'}(q)\times T_d^w$ is a $d$-split Levi subgroup of $G_n(q)$
(with a torus $T_d\cong \GL_1(q^d)$ if $d$ is odd, respectively
$T_d\cong \GU_1(q^{d'})$ if $d=2d'$ is even), and $\la$ is a $d$-cuspidal
unipotent character of $L$, and hence of $G_{n-wd'}(q)$ (see
\cite[Thm.]{CE94}). We then write $b(L,\la)$ for this block, and call $w$
its \emph{weight}. The unipotent characters in the block $b(L,\la)$ are then
the members of the $d$-Harish-Chandra series above $(L,\la)$, so by
\cite[Thm.~3.2]{BMM} they are in bijection with the irreducible characters of
the relative Weyl group of this $d$-cuspidal pair, which in this case is the
imprimitive complex reflection group $G(2d',1,w)$. In particular their number
is given by $k(2d',w)$.

The following result bears a strong resemblance to \cite[Prop.~6]{Ol84};
and again it expresses $k(B)$ only in terms of the weight $w$ of $B$, of
$d=d_\ell(q)$ and the $\ell$-part of $q^d-1$:

\begin{prop}   \label{prop:kB BC}
 Let $G\in\{\Sp_{2n}(q),\SO_{2n+1}(q)\}$, let $\ell\ne2$ and $B$ be a
 unipotent $\ell$-block of $G$ of weight~$w$. Let $d=d_\ell(q)$,
 $d':=d/\gcd(d,2)$ and write $\ell^a$ for the precise power of $\ell$ dividing
 $q^d-1$. Then
 $$k(B)=\sum_{\bw} k\big(2d'+(\ell^a-1)/2d',w_0\big)
         \prod_{i\ge1}k\big((\ell^a-\ell^{a-1})/2d',w_i\big),$$
 where the sum runs over all sequences $\bw=(w_0,w_1,\ldots)$ of non-negative
 integers satisfying
 $$\sum_{i\ge0} w_i\ell^i=w.$$
\end{prop}

\begin{proof}
First assume that $G=\SO_{2n+1}(q)$. We count the characters in $B$
lying in Lusztig series $\cE(G,t)$, with $t\in G^*=\Sp_{2n}(q)$ an
$\ell$-element. The conjugacy classes of semisimple elements in $\Sp_{2n}(q)$
are uniquely determined by the characteristic polynomials of their elements
in the natural $2n$-dimensional matrix representation. Furthermore, there is a
semisimple element with characteristic polynomial a given monic polynomial
$f\in\FF_q[X]$ of degree $2n$ with non-zero constant coefficient if and only
if any root of $f$ in $\overline{\FF}_q$ has the same multiplicity as its
inverse. That is to say, $f$ must be invariant under the transformation
$$f\mapsto f^*:=X^{\deg(f)}f(X^{-1})/f(0).$$
Let $\cF$ denote the set of monic irreducible polynomials $f\in\FF_q[X]$ whose
roots in $\overline{\FF}_q$ are of $\ell$-power order. We choose a system of
representatives $\bar\cF$ of $^*$-orbits in $\cF$. Write $\cF^i$ for the subset
of polynomials in~$\cF$ whose roots have order $\ell^i$, $i\ge0$. Then
$f\in\cF$ has degree
$$c_f:=\deg(f)=\begin{cases}
   1& \text{if $f\in \cF^0=\{X-1\}$},\\
   d& \text{if $f\in \cF^i$ with $1\le i\le a$},\\
   d\ell^i& \text{if $f\in \cF^{a+i}$ with $i\ge1$.}\end{cases}
$$
Now first assume that $d$ is odd. Then no $1\ne\al\in\overline{\FF}_q^\times$
of $\ell$-power order is Galois conjugate to its inverse, as otherwise
$\al^{q^c}=\al^{-1}$, that is, $\al^{q^c+1}=1$ for some $c\ge0$, which is
absurd. Hence, all orbits of $^*$ on $\cF^i$, $i\ge1$, have length~2.
Thus the classes of $\ell$-elements $t$ in $G^*$ are in bijection with maps
$m:\bar\cF\rightarrow\ZZ_{\ge0}$, $f\mapsto m_f$, with
$\sum_{f\in\bar\cF} m_f\deg(f)=n$, such that $t$ has characteristic polynomial
$(X-1)^{2m_1}\prod_{X-1\ne f\in\bar\cF}(f f^*)^{m_f}$, where for notational
convenience we write $m_1:=m_{X-1}$. The centraliser $C=C_{G^*}(t)$ of
$t\in G^*$ corresponding to the map $m$ is then isomorphic to
$\Sp_{2m_1}(q)\times\prod_{X-1\ne f\in\bar\cF} \GL_{m_f}(q^{c_f})$.
\par
According to \cite[Thm.]{CE94} a character of $\cE(G,t)$ lies in the block $B$
parametrised by $(L,\la)$ if its Jordan corresponding unipotent character of
$C$ lies in the unipotent block $B_C$ corresponding to the same pair $(L,\la)$.
For this to happen we must have in particular that $L$ is isomorphic to a
$d$-split Levi subgroup of $C^*$, hence $m_1=n-(w-u)d$ for some $u\ge0$. The
unipotent characters of $C$ in $B_C$ are then the outer tensor products of
unipotent characters in the block corresponding to $(L,\la)$ in $\Sp_{2m_1}(q)$
times arbitrary unipotent characters in the other factors $\GL_{m_f}(q^{c_f})$,
so their number is given by $k(2d,u)\prod_{f\ne X-1}\pi(m_f)$ with
$\pi(m_f)=k(1,m_f)$ the number of partitions of $m_f$. Now clearly the number
of elements in $\bar\cF^{a+i}$ is $(\ell^a-\ell^{a-1})/(2d)$. Then the
combinatorial argument in \cite[p.~45]{Ol84}, with $2d$ replacing $e'$,
applies to show the stated formula for $k(B)$.
\par
Next assume that $d=2d'$ is even. Then all $\ell$-elements in
$\overline{\FF_q}^\times$ are Galois conjugate to their inverses, and so
$f=f^*$ for all $f\in\cF$. Thus the classes of $\ell$-elements $t\in G^*$ are
in bijection with maps $m:\cF\rightarrow\ZZ_{\ge0}$, $f\mapsto m_f$, with
$\sum_{f\in \cF} m_f\deg(f)=2n$, such that $t$ has characteristic polynomial
$(X-1)^{m_1}\prod_{X-1\ne f\in \cF}f^{m_f}$, where again we write
$m_1:=m_{X-1}$. The centraliser $C=C_{G^*}(t)$ of $t\in G^*$ corresponding
to the map $m$ is then isomorphic to
$\Sp_{m_1}(q)\times\prod_{X-1\ne f\in \cF} \GU_{m_f}(q^{c_f/2})$.
Again by \cite[Thm.]{CE94} a character of $\cE(G,t)$ lies in the block $B$
parametrised by $(L,\la)$ if its Jordan corresponding unipotent character of
$C$ lies in the unipotent block $B_C$ corresponding to $(L,\la)$.
In that case, $L$ is isomorphic to a $d$-split Levi subgroup of $C$, hence
$m_1=2n-2(w-u)d'$ for some $u\ge0$. As before the unipotent characters of $C$
in $B_C$ are the outer tensor products of unipotent characters in the block
corresponding to $(L,\la)$ in $\Sp_{m_1}(q)$ times arbitrary unipotent
characters in the other factors $\GU_{m_f}(q^{c_f/2})$, so their number is
given by $k(d,u)\prod_{f\ne X-1}\pi(m_f/2)$. The number of elements in
$\cF^{a+i}$ is $(\ell^a-\ell^{a-1})/d$. Again we conclude as in
\cite[p.~45]{Ol84}, using that $d=2d'$.
\par
To treat $G=\Sp_{2n}(q)$ we need to consider $\ell$-elements in the dual
group $\SO_{2n+1}(q)$. But according to \cite[Prop.~4.2]{GH91} there is a
bijection between conjugacy classes of $\ell$-elements in $G$ and $G^*$ sending
centralisers to their duals. Since centralisers of odd order elements in both
$G,G^*$ are always connected, they have the same numbers of unipotent
characters. So the count for $\Sp_{2n}(q)$ is exactly the same as for
$\SO_{2n+1}(q)$ and we are done.
\end{proof}

\subsection{Even-dimensional orthogonal groups}   \label{subsec:ortho}
We next consider the even dimensional orthogonal groups.
Let $G_n^\eps(q)=\SO_{2n}^\eps(q)$, with $\eps\in\{\pm\}$, $n\ge4$. (Here, as customary, we write $\SO_{2n}$ for the connected component of the identity in
the general orthogonal group $\GO_{2n}$.) We recall some facts on blocks of
$G_n^\eps(q)$ from \cite[Thm.]{CE94}. Let $\ell$ be an odd prime and
$d=d_\ell(q)$. The unipotent $\ell$-blocks of $G_n^\eps(q)$ are again
parametrised by $d$-cuspidal pairs $(L,\la)$, where
$L=G_{n-wd'}^\delta(q)\times T_d^w$, with either $T_d=\GL_1(q^d)$ for odd
$d=d'$, or $T_d=\GU_1(q^{d'})$ for $d=2d'$ even,
and $\delta=\eps$ if $d$ is odd or $w$ is even, and $\delta=-\eps$ else, and
$\la$ is a $d$-cuspidal unipotent character of $L$. We write $B=b(L,\la)$ for
the corresponding block; and call $w$ the \emph{weight} of $B$. A defect group
of $b(L,\la)$ is then obtained as a Sylow $\ell$-subgroup of
$C_G([\bL,\bL])$, which in our case is $\GO_{2wd'}^{\eps\delta}(q)$. Observe
that by the parity condition on the sign $\eps\delta$, a Sylow $\ell$-subgroup
of $\GO_{2wd'}^{\eps\delta}(q)$ is already a Sylow
$\ell$-subgroup of $\SO_{2wd'+1}(q)$. The number of unipotent characters of
$G_n^\eps(q)$ in the $\ell$-block $B$ then equals the number of irreducible
characters of the relative Weyl group of $(L,\la)$, which is $G(2d',1,w)$
unless $\la$ is parametrised by a degenerate symbol, in which case it is the
normal subgroup $G(2d',2,w)$ (see \cite[p.~51]{BMM}).   \par
We first derive a closed formula for the number of characters in blocks of
the (disconnected) general orthogonal groups. Let $\tilde B$ be a block of
$\GO_{2n}^\eps(q)$ lying above the unipotent block $B=b(L,\la)$ of
$\SO_{2n}^\eps(q)$. Then either $\tilde B$ lies above a unique unipotent block
of $\SO_{2n}^\eps(q)$, in which case the tensor product of $\tilde B$ with the
non-trivial linear character of $\GO_{2n}^\eps(q)$ is another block above $B$,
or else the cuspidal pair $(L,\la)$ is such that $\la$ is labelled by a
degenerate symbol, in which case $\tilde B$ lies above the two blocks
parametrised by the two unipotent characters labelled by this degenerate symbol.
In either case, the unipotent characters in $\tilde B$ are in bijection with
the irreducible characters of $G(2d',1,w)$, which is the relative Weyl group
in $\GO_{2n}^\eps(q)$ of the cuspidal pair $(L,\la)$.

\begin{prop}   \label{prop:kB D}
 Let $\ell\ne2$ and $B$ be an $\ell$-block of $\GO_{2n}^\eps(q)$, $n\ge4$,
 lying above a unipotent $\ell$-block of $\SO_{2n}^\eps(q)$ of weight $w$.
 Let $d=d_\ell(q)$, $d':=d/\gcd(d,2)$ and write $\ell^a$ for the precise
 power of $\ell$ dividing $q^d-1$. Then
 $$k(B)= \sum_{\bw} k\big(2d'+(\ell^a-1)/2d',w_0\big)
         \prod_{i\ge1}k\big((\ell^a-\ell^{a-1})/2d',w_i\big),$$
 where the sum runs over all sequences $\bw=(w_0,w_1,\ldots)$ of non-negative
 integers satisfying
 $$\sum_{i\ge0} w_i\ell^i=w.$$
\end{prop}

\begin{proof}
Let $\bG=\SO_{2n}\le\hat\bG:=\GO_{2n}\le\tilde\bG:=\SO_{2n+1}$ embedded in the
natural way and write $G=\bG^F$, $\hat G=\hat\bG^F$ and $\tG=\tilde\bG^F$.
Let $B$ be a unipotent $\ell$-block lying above the block $b(L,\la)$ of $G$ of
weight $w$, for a $d$-cuspidal unipotent character of some $d$-split
Levi subgroup $\bL\le\bG$. To investigate $\cE_\ell(G,1)\cap\Irr(B)$, let
$t\in G^*$ be an $\ell$-element. Let $(m_f)_{f\in\cF}$ describe the
characteristic polynomial of $t$ in the natural matrix representation of
$G^*\cong G$. Then the centraliser of $t$ in $G^*$ has the same form as in
$\tilde G=\SO_{2n+1}(q)$, except that the factor $G_{n-wd'}(q)$ is replaced
by a group $\GO_{2(n-wd')}^\delta(q)$. Now again by \cite[Thm.]{CE94} such an
$\ell$-element $t$ contributes to $B$ if and only if $C_{\bG^*}(t)$ contains
$\bL^*$, so if
$$C_{\hat G^*}(t)\cong
  \GO_{2m_1}^\delta(q)\times\prod_{X-1\ne f\in\cF}\GL_{m_f}(q^{c_j})$$
with $m_1=n-(w-u)d'$ for some $u\ge0$.
\par
Now first assume that $d$ is odd. Then by the remarks preceding this
proposition the number of unipotent characters of the first factor of this
centraliser in the $\ell$-block above $\la$ is given by
$|\Irr(G(2d,1,u))|=k(2d,u)$, just as in the case of $\SO_{2n+1}(q)$, and so
the number of unipotent characters above $\la$ 
in this centraliser is given by $k(2d,u)\prod_{f\ne X-1}\pi(m_f)$, exactly as
in the previous proof. Now note that any two $\ell$-elements with this shape
of centraliser lie inside a subgroup $\GO_{2m_1}^\delta(q)\times\GL_{m_0}(q)$,
with $m_0=n-m_1$, and thus are conjugate already inside this group. So the
classes of $\ell$-elements with Lusztig series contributing to $B$ are
parametrised exactly as in $\tilde G^*$. Thus, we see that $k(B)$ is given by
the same expression as the one we obtained in Proposition~\ref{prop:kB BC} for
$\SO_{2n+1}(q)$. The case of even $d$ is entirely similar.
\end{proof}

To descend to the special orthogonal groups, we need an auxiliary result on
characters of certain imprimitive complex reflection groups.

\begin{lem}   \label{lem:irrs}
 Let $d,n\ge1$. Then $|\Irr(G(2d,2,n))|\le|\Irr(G(2d,1,n))|=k(2d,n)$.
\end{lem}

\begin{proof}
The irreducible characters of the complex reflection group $G(2d,1,n)$,
which is the wreath product $C_{2d}\wr\fS_n$, are naturally indexed by
$2d$-tuples of partitions of~$n$. This shows that $|\Irr(G(2d,1,n))|=k(2d,n)$.
An irreducible character of $G(2d,1,n)$ parametrised by such a $2d$-tuple
$(\la_1,\ldots,\la_{2d})$ stays irreducible upon restriction to the normal
subgroup $G(2d,2,n)$ of index~2 unless $\la_i=\la_{i+d}$ for $i=1,\ldots,d$.
In this case, $(\la_1,\ldots,\la_d)$ is a $d$-tuple of partitions of $n/2$,
and in particular $n$ must be even, which we assume from now on. In this case
our preceding discussion gives
$|\Irr(G(2d,2,n))|=(k(2d,n)-k(d,n/2))/2+2k(d,n/2)$. So our claim is proven
when we can show that $k(2d,n)\ge 3k(d,n/2)$. \par
Now according to \cite[Lemma~1(ii)]{Ol84} we have
$$k(2d,n)=\sum_{1\le t\le n} k(d,t)k(d,n-t)\ge k(d,n/2)^2,$$
so we are done if $k(d,n/2)\ge3$. Using \cite[Prop.~5]{Ol84}, for example,
one sees that this holds unless $(d,n/2)\in\{(1,1),(1,2),(2,1)\}$. For these
three cases the claim can be checked directly.
\end{proof}

\begin{cor}   \label{cor:kB D}
 Let $G=\SO_{2n}^\pm(q)$ with $n\ge4$, let $\ell\ne2$ and $B$ be a unipotent
 $\ell$-block of $G$ of weight~$w$. Let $d=d_\ell(q)$, $d':=d/\gcd(d,2)$ and
 write $\ell^a$ for the precise power of $\ell$ dividing $q^d-1$. Then
 $$k(B)\le \sum_{\bw} k\big(2d'+(\ell^a-1)/2d',w_0\big)
         \prod_{i\ge1}k\big((\ell^a-\ell^{a-1})/2d',w_i\big),$$
 where the sum runs over all sequences $\bw=(w_0,w_1,\ldots)$ of non-negative
 integers satisfying
 $$\sum_{i\ge0} w_i\ell^i=w.$$
\end{cor}

As the proof shows, the inequality will often be strict, but this form will
suffice for our purpose.

\begin{proof}
Let $\bG=\SO_{2n}\le\hat\bG:=\GO_{2n}$ embedded
in the natural way. Let $B$ be a unipotent $\ell$-block of $G=\bG^F$ and let
$\hat B$ be an $\ell$-block of $\hat G=\hat\bG^F$ lying above it. First assume
that $\hat B$ covers two blocks of $G$. Then clearly $k(\hat B)=k(B)$ and we
are done by Proposition~\ref{prop:kB D}. So $B$ is the only block covered by
$\hat B$. In the proof of Proposition~\ref{prop:kB D} we determined the
contribution of the various Lusztig series $\cE(G,t)$, $t\in G^*$ an
$\ell$-element, to $\Irr(\hat B)$, in terms of the unipotent characters of the
centraliser $C_{G^*}(t)$. Now the number of unipotent characters in a given
$\ell$-block of a factor $\SO_{2m_1}^\pm(q)$, with $m_1=n-(w-u)d'$ for some
$u\ge0$, is given by $|\Irr(G(2d,2,u))|$, while the number of characters lying
above it in $\GO_{2m_1}^\pm(q)$ equals $|\Irr(G(2d,1,u))|$. According to
Lemma~\ref{lem:irrs} this second number is always at least as big as the
former, so any Lusztig series $\cE(G,t)$ contributes at most as many
characters to $B$ as the characters above it contribute to $\hat B$, whence
$k(B)\le k(\hat B)$. Thus our claim follows with Proposition~\ref{prop:kB D}.
\end{proof}

\begin{thm}   \label{thm:class}
 Let $H$ be quasi-simple of classical Lie type in characteristic~$p$ and
 assume that $\ell\ne2,p$. Then the unipotent $\ell$-blocks of $H$ are not
 minimal counterexamples to the strong form of Brauer's $k(B)$-conjecture.
\end{thm}

\begin{proof}
By Proposition~\ref{prop:spor} we do not have to consider exceptional covering
groups. As the order of the non-exceptional part of the Schur multiplier
of $H/Z(H)$ is a power of~2, and all unipotent blocks have $Z(H)$ in their
kernel, we in fact only need to consider the case when $H$ is simple. Moreover
by Lemma~\ref{lem:ab def} we can restrict attention to unipotent blocks
with non-abelian defect groups.
\par
Let us first consider $G=\SO_{2n+1}(q)$ with $n\ge2$. Let $d=d_\ell(q)$ and let
$B$ be a unipotent $\ell$-block of $G$ parametrised by the $d$-cuspidal pair
$(L,\la)$, with $L$ of semisimple rank $n-wd'$ where $d'=d/\gcd(d,2)$. By
\cite[Thm.~4.4(ii)]{CE94} the defect groups of $B$ are isomorphic to Sylow
$\ell$-subgroups of $C_G([\bL,\bL])$. Now $[\bL,\bL]=\SO_{2(n-wd')+1}$
has centraliser $\GO_{2wd'}^\pm(q)$ in $G$, where the ``+'' sign occurs
if and
only if $d$ is odd. A Sylow $\ell$-subgroup of $\GO_{2wd'}^\pm(q)$ is
isomorphic to the wreath product $C_{\ell^a}\wr P$, with $\ell^a$ the precise
power of $\ell$ dividing $q^d-1$ and $P$ a Sylow $\ell$-subgroup of the complex
reflection group $G(2d',1,w)$. On the other
hand, the number $k(B)$ was computed in Proposition~\ref{prop:kB BC}.
Application of the precise same estimates as in \cite[p.~46]{Ol84} now shows
that $k(B)< |D|$ whenever $D$ is non-abelian (which we may assume by
Lemma~\ref{lem:ab def}). Since the
simple group $H=[G,G]$ has index at most~2 in $G$, and restriction of
characters in $\ell$-series of $G$ is irreducible by Lusztig's parametrisation
(see \cite[Prop.~15.6]{CE}), we obtain our claim for the unipotent blocks
of $H$.
\par
We can argue as in the previous case to obtain the desired
inequality for unipotent blocks of $G=\Sp_{2n}(q)$, just replacing the
centraliser of $[\bL,\bL]=\Sp_{2(n-wd')}$ in $G$ by $\Sp_{2wd'}(q)$, whose
Sylow $\ell$-subgroups have the same order. Then the claim also holds for
the unipotent blocks of the simple factor group $\PSp_{2n}(q)$.
\par
Finally, consider $G=\SO_{2n}^\pm(q)$, $n\ge4$. Let $B$ be a unipotent
$\ell$-block of $G$ parametrised by the $d$-cuspidal pair $(L,\la)$. As
observed above, we may assume that $B$ has non-abelian defect, that is,
$w\ge\ell$. It is now easy to check that the bound for $k(B)$ given in
Corollary~\ref{cor:kB D} is less than $|D|$ in all cases. Now the derived
subgroup $[G,G]$ has index at most~2 in $G$, and characters corresponding to
$\ell$-elements restrict irreducibly, and the simple group $H$ is obtained as
$[G,G]Z(G)/Z(G)$, with $|Z(G)|\le2$, which completes the proof. 
\end{proof}

\section{Exceptional groups in non-defining characteristic} \label{sec:exc}
We now deal with the unipotent blocks of exceptional groups of Lie type.
The strong form of the $k(B)$-conjecture for  $\tw3D_4(q)$ was shown in
\cite{DM87}, and for $G_2(q)$ the statement can be read off from \cite{Hi}.
So in view of Lemma~\ref{lem:suzree} we only need to concern ourselves with
the five series of exceptional groups of rank at least~4. Continuing our
previous notation let $\bG$ be simple of simply connected exceptional type and
of rank at least~4 with a Frobenius endomorphism $F$, and $G=\bG^F$. In this
section we also allow $\ell=2$.

\subsection{Unipotent blocks} \label{subsec:exc unip}

\begin{prop}   \label{prop:principal exc}
 The principal $\ell$-block of a quasi-simple exceptional group of Lie type
 $G/Z$, where $Z\le Z(G)$, is not a counterexample to the $k(B)$-conjecture
 in strong form.
\end{prop}

\begin{proof}
We first consider the assertion for the principal $\ell$-block $B_0$
of~$G$. Then clearly it will follow for the principal $\ell$-block of $G/Z$
for all $Z\le Z(G)$ such that $\ell$ does not divide $|Z|$. 
Note that unless $\ell$ divides the order of the Weyl group of $G$, the Sylow
$\ell$-subgroups of $G$ are abelian and $\ell$ is good for $\bG$ (see
\cite[Thm.~25.14]{MT}). In the latter case, in particular $\ell$ does not
divide the order of the center of $G$. So $B_0$ is isotypic to a block of the
normaliser of a Sylow $\ell$-subgroup of $G$ by Lemma~\ref{lem:ab def}, hence
of an $\ell$-solvable group, for which the
$k(B)$-conjecture has been shown to hold. 
\par
So we may assume that $\ell$ divides the order of the Weyl group of $\bG$. Let
$\bG^*$ be dual to $\bG$ and
set $G^*=\bG^{*F}$. According to \cite[Thm.~5.14]{MH0} the normaliser of a
Sylow $\ell$-subgroup $P$ of $G^*$ embeds into the normaliser $N:=N_{G^*}(\bT)$
of an $F$-stable maximal torus $\bT$ of $\bG^*$ containing a Sylow
$d$-torus of $\bG^*$, where $d=d_\ell(q)$. Now the principal $\ell$-block $B_0$
of $G$ satisfies
$$\Irr(B_0)\subseteq\cE_\ell(G,1)=\bigcup_t\cE(G,t)$$
(see~(*) in Section~\ref{sec:lie-nondef}), where $t$ runs over a system of
representatives of the conjugacy
classes of $\ell$-elements in $G^*$. Furthermore, by Lusztig's Jordan
decomposition for any such $t$, $\cE(G,t)$ is in bijection with
$\cE(C_{G^*}(t),1)$, the unipotent characters of $C_{G^*}(t)$. Since every
conjugacy class of $\ell$-elements of $G^*$ has a representative in our
chosen Sylow $\ell$-subgroup $P$, it hence suffices to show that
$$|P|>\sum_{t\in P/\sim} |\cE(G,t)|
  =\sum_{t\in P}\frac{1}{|t^N\cap P|}\,|\cE(C_{G^*}(t),1)|,$$
where the first sum runs over a system of representatives of elements in $P$
modulo $N$-conjugation, and the second over all elements of $P$. In particular
the right hand side has exactly $|P|$ summands. Hence we are done if we show
that the average value of the summands is smaller than~1.
\par
For this we split the
sum into two parts, depending on whether $t\in P\cap\bT$ or not.
The centraliser of a semisimple element $t\in \bT$ is a subsystem subgroup
of $\bG^*$, with Weyl group $W_t$ a reflection subgroup of the relative Weyl
group $W=N_G(\bT)/\bT^F$ of $\bT$ (see \cite[Thm.~14.2]{MT}). Thus
$|t^N\cap P|\ge|W:W_t|$ and we are done whenever $|W:W_t|>|\cE(C_{G^*}(t),1)|$.
The unipotent characters have been classified by Lusztig; in particular he has
shown that for a connected reductive group their number is multiplicative over
the simple
components. For $\bG$ of type $A_n$, $|\cE(\bG^F,1)|$ is the number of
partitions of $n+1$, for the other simple groups of low rank these numbers are
given in Table~\ref{tab:numunip} (see e.g. \Chevie\ \cite{Mi15}).

\begin{table}[htb]
\caption{Numbers of unipotent characters for simple groups}   \label{tab:numunip}
$$\begin{array}{c|ccccccccc}
 G& B_2,C_2& B_3,C_3& B_4,C_4& D_4& \tw2D_4& F_4& D_5,\tw2D_5\\
\noalign{\hrule}
 |\cE(G,1)|& 6& 12& 25& 14& 10& 37& 20\\
\end{array}$$
$$\begin{array}{c|ccccccccc}
 G& D_6& \tw2D_6& E_6,\tw2E_6&  D_7,\tw2D_7& E_7& D_8& \tw2D_8& E_8\\
\noalign{\hrule}
 |\cE(G,1)|& 42& 36& 30& 65& 76& 120& 110& 166\\
\end{array}$$
\end{table}

For $t\in P\setminus\bT$ let $\bar M=N_N(P)\bT^F/\bT^F$ and $\bar t$ the image
in $\bar M$ of $t$. Then $t$ has at least 
$|\bT\cap P|/|C_{\bT\cap P}(t)|\cdot |\bar M|/|C_{\bar M}(\bar t)|$ conjugates
in $P$, and again we are done if this number exceeds $|\cE(C_{G^*}(t),1)|$.
Note that $\bar t$ is a non-trivial $\ell$-element of $W$. The value of this
bound on class lengths can easily be computed inside $W$.

Let first $G=F_4(q)$. Then only $\ell\le 3$ needs to be considered. Here $\bT$
is either maximally split with $|\bT^F|=(q-1)^4$ if $d=1$, or Ennola dual to
that with $|\bT^F|=(q+1)^4$ if $d=2$. For $\ell=2$ we compute the orbits of
$W$ on the set of elements $t$ of order $o(t)$ diving $4$ in $\bT$ using
\Chevie\ \cite{Mi15}; the results are as follows:
$$\begin{array}{lc|cc}
 C_{G^*}(t)& o(t)& |\cE(G,t)|& |P\cap t^N|\\
\noalign{\hrule}
               F_4(q)& 1& 37&  1\\
               B_4(q)& 2& 25&  3\\
        C_3(q).A_1(q)& 2& 24& 12\\
         B_3(q).(q-1)& 4& 12& 24\\
         C_3(q).(q-1)& 4& 12& 24\\
 A_3(q).\tilde A_1(q)& 4& 10& 24\\
  B_2(q).A_1(q).(q-1)& 4& 12& 72\\
 \tilde A_2(q).A_1(q).(q-1)& 4& 6& 96\\
\noalign{\hrule}
\end{array}$$
Visibly, the average value of $|\cE(G,t)|/|P\cap t^N|$ is strictly less than~1
on the union of these classes. Now the involutions and the elements of order~4
with centraliser $A_3(q).\tilde A_1(q)$ are the only isolated 2-elements in
$\bG$, so any other 2-element in $P\cap \bT$ has centraliser of semisimple
rank at most~3, so at most 12 unipotent characters but at least 24 conjugates
in $P$, whence the average value over all 2-elements in $P\cap\bT$ is indeed
strictly less than~1.

For $\ell=3$ the classes of elements in $P\cap\bT$ of order dividing~3 are as
follows, 
$$\begin{array}{lc|cc}
 C_{G^*}(t)& o(t)& |\cE(G,t)|& |P\cap t^N|\\
\noalign{\hrule}
               F_4(q)& 1& 37&  1\\
         B_3(q).(q-1)& 3& 12& 24\\
         C_3(q).(q-1)& 3& 12& 24\\
 A_2(q).\tilde A_2(q)& 3&  9& 32\\
\noalign{\hrule}
\end{array}$$
which gives an average value less than~1. Furthermore, $W(F_4)$ has only three
classes of non-trivial elements of $3$-power order, with orbit length at least
$18$ in $PW(F_4)$, larger than the number of unipotent characters
of any proper centraliser. 

Now let $G=E_6(q)$. Here we need to consider $\ell=2,3$ and moreover $\ell=5$
when $5|(q-1)$. Again, the relevant classes of elements of order 4, 3 and~5
can be computed with \Chevie, and the claim follows for the group $G$. If
$Z\ne1$ then we have $3|(q-1)$, $|Z|=3$ and we may assume that $\ell=3$.
The classes of elements of $G^*$ of order dividing~3 contained in $P\cap \bT$
are
$$\begin{array}{l|cc}
 C_{G^*}(t)& |\cE(G,t)|& |P\cap t^N|\\
\noalign{\hrule}
                 E_6(q)& 30& 1\\
           A_5(q).(q-1)& 11& 72\\
             A_2(q)^3.3& 17& 80\\
       D_4(q).(q-1)^2.3& 26& 30\\
 \tw3D_4(q).(q^2+q+1).3& 24& 60\\
           D_5(q).(q-1) \text{ (twice)}& 20& 27\\
    A_4(q).A_1(q).(q-1) \text{ (twice)}& 14& 216\\
\noalign{\hrule}
\end{array}$$
and again it follows that the average value under consideration is less than~1.
The situation for the twisted groups $\tw2E_6(q)$ is entirely similar except
that $\ell=5$ now has to be considered when $5|(q+1)$.
\par
For $G=E_7(q)$ we need to consider $\ell=2,3$, and $\ell=5,7$ when either
divides $q^2-1$. The only centralisers for which the number of unipotent
characters is larger than the index of the corresponding Weyl group in $W(E_7)$
are those
of types $E_7$, $E_6.2$ and $D_6+A_1$, which correspond to quasi-isolated
involutions and those only occur once each. Adding the contributions by the
other centralisers of elements of order~4, 3, 5 or~7 respectively, we again get
an average value below~1. For $G/Z$ with $|Z|=2$, we need to consider
$\ell=2$. Here the claim follows by an analogous computation.
\par
Finally, for $G=E_8(q)$ we need to consider $\ell=2,3,5$, and $\ell=7$ when
$7|(q^2-1)$. Here, only $t=1$ and the class of involutions with centraliser
of type $E_7+A_1$ give a too large contribution, but this is again offset by
the collection of all elements of order~4.
\end{proof}

We now turn to general unipotent blocks. At bad primes these were determined 
by Enguehard \cite{En00}.

\begin{thm}   \label{thm:unip exc}
 Let $B$ be a unipotent $\ell$-block of a quasi-simple exceptional group of
 Lie type $G/Z$, where $Z\le Z(G)$. Then $B$ is not a minimal counterexample
 to the $k(B)$-conjecture in strong form.
\end{thm}

\begin{proof}
By Proposition~\ref{prop:principal exc} we may assume that $B$ is not the
principal block of $G/Z$, and by Lemma~\ref{lem:ab def} we only need to
consider the case of bad primes $\ell$. Furthermore, by Theorem~\ref{thm:sam}
we may assume that defect groups have rank at least~4 when they are abelian.
For type $F_4$, by \cite[p.~349]{En00} the only non-principal unipotent block
of positive defect occurs for $\ell=3$, but it has abelian defect groups of
rank~2.    \par
For $G=E_6(q)$ and $\ell\le5$ again by
\cite[p.~351]{En00} the only relevant unipotent block is the one above the
cuspidal unipotent character of $D_4(q)$ for $\ell=3$ and $q\equiv1\pmod3$,
with defect group $D$ of order~$3(q-1)_3^2$. For this non-principal
3-block $B$ we use the description of $\Irr(B)$ given in \cite[Thm.~B]{En00}:
we have $\Irr(B)\subseteq\cE_\ell(G,1)$, and for any $\ell$-element $t\in G^*$,
$\Irr(B)\cap\cE(G,t)$ is in bijection with the irreducible characters in a
corresponding unipotent block of $C_{G^*}(t)$, belonging to the ``same''
$1$-cuspidal pair (apart from certain exceptions as described in
\cite[Prop.~17]{En00}). In our situation, $B$ is associated to the $1$-cuspidal
pair $(D_4(q),\la)$, where $\la$ denotes the unipotent cuspidal character of
$D_4(q)$, and thus $\cE(G,t)$ with $t$ a non-trivial 3-element contains a
character from $B$ only if $t$ has centraliser containing a Levi subgroup of
type $D_4$. Thus the possibilities for $C_{G^*}(t)$ are $D_4(q).(q-1)^2$,
$D_5(q).(q-1)$ and $\tw3D_4(q).(q^2+q+1)$, in which case we have
$|\Irr(B)\cap \cE(G,t)|=1,2,1$ respectively (in the last case corresponding to
the cuspidal unipotent character $\tw3D_4[-1]$ of $\tw3D_4(q)$, see loc.~cit.).
As $|\Irr(B)\cap\cE(G,1)|=3$, and any $t$ as before has at least~6 conjugates
in $P$, the claim ensues.

The arguments for the non-principal unipotent blocks of $\tw2E_6(q)$, described
in \cite[p.~354]{En00}, are entirely similar except that here $d=d_3(q)=2$.

For $G=E_7(q)$ by the table in \cite[p.~354]{En00}, the only relevant blocks
are for $\ell=2$ the blocks lying above the $d$-cuspidal unipotent characters
of Levi subgroups of type $E_6$, and for $\ell=3$ a block lying above the
$d$-cuspidal unipotent
character of a Levi subgroup of type $D_4$. In the case $\ell=2$ the defect
groups are dihedral (see \cite[p.~357]{En00}) in which case our claim is known
to hold, see \cite[Cor.~8.2]{S17}. For the non-principal unipotent 3-block
$B$ above $D_4$ the defect groups have order $|D|=3(q-\eps)^3$ where
$q\equiv\eps\pmod3$ (see loc.~cit.), and again we can compute the number of
characters in $\Irr(B)$ using \cite[Thm.~B]{En00}. The centralisers of
3-elements $1\ne t\in G^*$ containing a subgroup of type $D_4$ are of types
$$D_4,\ \tw3D_4,\ D_5,\ D_4+A_1,\ D_6,\ D_5+A_1,\text{ and }E_6$$
when $d=1$, or their Ennola duals for $d=2$,
with $|\Irr(B)\cap\cE(G,t)|=1,1,2,2,5,4,3$ respectively, but all of these have
at least 28 conjugates in $P$; this also makes up for the 10 characters
contributed by $\Irr(B)\cap\cE(G,1)$.
\par
For $G=E_8(q)$ we need to consider $\ell=2,3,5$. The relevant unipotent blocks
for $\ell=2$ are those above the $d$-cuspidal unipotent characters of $E_6(q)$
(or $\tw2E_6(q)$), and for $\ell=3,5$ those above the $d$-cuspidal unipotent
character of $D_4(q)$. The non-principal unipotent block for $\ell=2$ has
defect groups of
order $4(q-\eps)_2^2$ (see \cite[p.~364]{En00}) with $q\equiv\eps\pmod4$,
and only the Lusztig series $\cE(G,t)$ with $t$ a 2-element with centraliser
of type $\tw{(2)}E_6$, $E_7$ or $\tw{(2)}E_6+A_1$ will contribute, with
$|\Irr(B)\cap\cE(G,t)|=1,2,2$ respectively, in addition to the six unipotent
characters. Again, the inequality is easily seen to hold. The computations
for $\ell=3,5$ are similar to those done previously. This completes the
discussion of all unipotent blocks of exceptional type groups.
\end{proof}

\subsection{Isolated 5-blocks in $E_8(q)$} \label{subsec:E8}
The only simple groups of Lie type for which $5$ is a bad prime are those of
type $E_8$. In view of Proposition~\ref{prop:Enguehard} and
Theorem~\ref{thm:unip exc}, in order to complete the proof of Theorem~1 for
$p=5$ it therefore remains to treat the isolated 5-blocks of $E_8(q)$. These
have been classified in \cite[Prop.~6.10 and~6.11]{KM13}, but without
determining the precise character distribution.

\begin{prop}   \label{prop:E8 l=5}
 Let $B$ be an isolated 5-block of $E_8(q)$. Then $B$ is not a minimal
 counterexample to the $k(B)$-conjecture in the strong form.
\end{prop}

\begin{proof}
By Theorem~\ref{thm:unip exc} we may assume that $B$ is not unipotent. Thus
$B$ lies in $\cE_5(G,s)$ for an isolated semisimple $5'$-elements $1\ne s$
of $G^*$. According to Lemma~\ref{lem:one block} we may assume that $\cE_5(G,s)$
is not a single 5-block.
Now by \cite[Tab.~7 and~8]{KM13} the only non-unipotent isolated 5-blocks of
$E_8(q)$ for which the defect groups are not abelian of rank at most~3 and for
which $\cE_5(G,s)$ is not a single 5-block are those collected in
Table~\ref{tab:E8 l=5}, where we have kept the numbering from loc.~cit.
The two unnumbered lines correspond to 5-blocks that had inadvertently been
omitted. All cases occur when $q\equiv1\pmod5$, that is, $d=d_5(q)=1$, except
for the one in the last line which occurs when $q\equiv2,3\pmod5$. The case of
$q\equiv4\pmod5$ is entirely analogous to the one where $q\equiv1\pmod5$.
In the last column we provide the size of a defect group $D$ as described in
\cite[Thm.~1.2(b)]{KM13}, where $5^a$ denotes the precise power of~5 dividing
$q^d-1$.

\begin{table}[htb]
\caption{Some isolated 5-blocks in $E_8(q)$}   \label{tab:E8 l=5}
$\begin{array}{|r|r|lll|}
\hline
 \text{No.}& C_{\bG^*}(s)^F& \bL&  \la&  |D|\\
\hline\hline
 1&       D_8(q)& \emptyset& 1& 5^{8a+1}\\
 2&             &     D_4& D_4[1]& 5^{4a}\\
\hline
 3& E_7(q)A_1(q)& \emptyset& 1& 5^{8a+1}\\
 4&             &     D_4& D_4[1]& 5^{4a}\\
\hline
 7& D_5(q)A_3(q)& \emptyset& 1& 5^{8a+1}\\
 8&             & D_4& D_4[1]& 5^{4a}\\
\hline
11& \tw2A_7(q)A_1(q)& A_1^3&  1& 5^{5a}\\
\hline
14&    \tw2A_8(q)& A_1^4&  1& 5^{4a}\\
\hline
16& E_6(q)A_2(q)& \emptyset& 1& 5^{8a+1}\\
17&              &     D_4& D_4[1]& 5^{4a}\\
\hline
19& \tw2E_6(q).\tw2A_2(q)& A_1^3&  1& 5^{5a}\\
20&                      & D_4&  (\tw2A_2,\phi_{21})& 5^{4a}\\
\hline
 & \tw2A_5(q).\tw2A_2(q)A_1(q)& A_1^3&  1& 5^{5a}\\
 &                       & D_4&  (\tw2A_2,\phi_{21})& 5^{4a}\\
\hline\hline
25&                D_8(q)& \Ph4^4&  1& 5^{4a}\\
\hline
\end{array}$
\end{table}

First, consider cases~1 and~2. Here $\cE_5(G,s)$ contains two 5-blocks, and the
precise subdivision of characters among these blocks is not known. In case~2,
all characters in the corresponding block $B_2$ must have degree divisible by
$|G:D|_5=5^2(q-1)_5^4=5^{4a+2}$. But this does not hold for characters in the
principal series of $G$. The only other Harish-Chandra series occurring in a
Lusztig series $\cE(G,st)$, with $t\in C_{G^*}(s)$ as 5-element, are the ones
above a cuspidal unipotent character of a split Levi subgroup of type $D_4$.
In particular, $\cE(G,st)$ only contributes
to $\Irr(B_2)$ if $C_{G^*}(st)$ contains $D_4(q)$, that is, if $t$ centralises
a subgroup $D_4(q)$ in $C_{G^*}(s)=D_8(q)$. The centraliser of $D_4(q)$ inside
$D_8(q)$ is a subgroup $D_4(q)$, as can be seen from the extended Dynkin
diagram, so our claim for $B_2$ follows again with Corollary~\ref{cor:kB D}.
The exact same reasoning applies in cases~8 and~17 (note that by
\cite[Tab.~7]{KM13} the characters in the Harish-Chandra series above a
cuspidal character of $E_6(q)$ do not lie in the 5-block corresponding to
case~17). To deal with the case~1 we claim that the total number of characters
in $\cE_5(G,s)$ is less
than $|D|$; for this note that for any possible centraliser $C_{G^*}(st)$ the
number of non-principal series characters in $\cE(G,st)$ is at most one sixth
of the total number of characters. So $|\cE_5(G,s)|$ is at most 6/5 of the
number obtained from the principal series characters, which we determined in
Corollary~\ref{cor:kB D}. Evaluation of that formula shows our claim. \par
Cases~3,7,11,14,16,19 and~25 are settled completely analogously.
\par
For the unnumbered lines, by Proposition~\ref{prop:isogeny} it suffices to
count the characters in a simply connected covering
$H=\SU_6(q)\times\SU_3(q)\times \SL_2(q)$ of the centraliser $C_{G^*}(s)$.
By Olsson's formula for $\GU_n(q)$ we have that
$$|\cE_5(\GU_6(q)\times\GU_3(q)\times \GU_2(q),1)|
  \le 5^{4a}\cdot 5^{2a}\cdot5^{2a}=5^{8a}.$$
Now by Lusztig's result \cite[Thm.~15.11]{CE} all characters of
$\GU_6(q)\times\GU_3(q)\times \GU_2(q)$ corresponding to $5$-elements restrict
irreducibly to $H$, and since $H$ has index divisible by $5^{3a}$ it follows
that 
$$|\cE_5(\SU_6(q)\times\SU_3(q)\times \SL_2(q),1)|\le 5^{8a}/5^{3a}=5^{5a}.$$
So the first block with this centraliser (which has abelian defect groups)
satisfies the $k(B)$-conjecture. The characters in the other block must have
degree divisible by $|G:D|_5=5^{4a+2}$, so they must lie above the cuspidal
unipotent character of the factor $\SU_3(q)$. Here, the count for the middle
factor is just $5^a$, and again the desired inequality follows. A similar
argument can be employed to deal with the remaining two cases~4 and~20.
\end{proof}

\subsection{Isolated 3-blocks with abelian defect} \label{subsec:l=3}

\begin{prop}   \label{prop:l=3}
 Let $B$ be an isolated 3-block of a quasi-simple group of Lie type with
 abelian defect groups. Then $B$ is not a minimal counterexample to the
 $k(B)$-conjecture.
\end{prop}

\begin{proof}
By Lemma~\ref{lem:ab def} and Theorems~\ref{thm:SLn} and~\ref{thm:unip exc} we
may assume that $B$ is not unipotent. Thus $B$ lies in $\cE_3(G,s)$ for an
isolated semisimple $3'$-element $1\ne s$ of $G^*$. Moreover, $G$ is not of
classical type by Proposition~\ref{prop:Enguehard}. According to
Lemma~\ref{lem:one block} we may assume that $\cE_3(G,s)$ is not a single
3-block. Now by \cite[Tab.~2, 3, 4 and~6]{KM13} the only non-unipotent
isolated 3-block $B$ with abelian defect groups of rank at least~4 for
which $\cE_3(G,s)$ is not a single 3-block occurs in $E_8(q)$ with
$C_{G^*}(s)$ of type $\tw2A_4(q)^2$. Now note that the other two blocks in
$\cE_3(G,s)$ have smaller defect groups. But then according to the main result
of \cite{KM13} the characters of $\cE_3(G,s)$ in $B$ are exactly those of
height zero, hence those lying in the principal Harish-Chandra series. For
these we showed the validity of the required inequality in
Theorem~\ref{thm:SLn}.
\end{proof}

We conclude by proving Theorems~1, 2 and~3. Let $(G,B)$ be a minimal
counterexample to Brauer's $k(B)$-conjecture in the strong form with $G$
quasi-simple. Then $G$ must be of Lie type by Proposition~\ref{prop:spor}.
Theorem~\ref{thm:lie-p} shows that $p$ is not the defining prime. The block $B$
is not unipotent for $p\ne2$ by Theorems~\ref{thm:SLn}, \ref{thm:class}
and~\ref{thm:unip exc}. It is shown in Lemma~\ref{lem:not qi} that $B$ must
be isolated. The isolated blocks for good primes $p\ge3$ are not minimal
counterexamples by Enguehard's result in Proposition~\ref{prop:Enguehard},
and the isolated 5-blocks of $E_8(q)$ are neither by
Proposition~\ref{prop:E8 l=5}. This achieves the proof of Theorem~1.
\par
Now consider Theorem~2 on blocks with abelian defect groups. In the case $p=3$
the claim follows with Proposition~\ref{prop:l=3}. Now assume that $B$ is a
quasi-isolated 2-block with abelian defect. Then by \cite[Lemma~5.2]{EKKS} we
have that either defect groups of $B$ have rank at most~2, or $G$ is of
type~$A$ and $B$ is quasi-isolated but not isolated. In this case we may
conclude by Lemma~\ref{lem:not qi}. Finally, the principal block of a group
of Lie type is unipotent, so the last assertion of Theorem~3 also follows.


\end{document}